\documentclass[a4paper,leqno,centertags,fleqn]{amsart}
\usepackage{amsmath}
\usepackage{amsfonts}

\newtheorem{theorem}{Theorem}
\theoremstyle{plain}

\newtheorem{corollary}{Corollary}

\newtheorem{definition}{Definition}
\newtheorem{example}{Example}

\newtheorem{lemma}{Lemma}

\newtheorem{remark}{Remark}

\numberwithin{equation}{section}

\begin{document}
\title[Quantum integral inequalities]{Quantum Montgomery identity and some
quantum integral inequalities}
\author[M. Kunt, A. Kashuri, T.-S. Du]{Mehmet Kunt$^{1}$, Artion Kashuri$^{2}
$ and Tingsong Du$^{3,4,\ast }$}
\address{$^{1}$Department of Mathematics, Faculty of Sciences, Karadeniz
Technical University, 61080, Trabzon, Turkey}
\email{mkunt@ktu.edu.tr}
\address{$^{2}$Department of Mathematics, Faculty of Technical Science,
University Ismail Qemali, Vlora, Albania}
\email{artionkashuri@gmail.com}
\address{$^{3}$Department of Mathematics, College of Science, China Three
Gorges University, 443002, Yichang, P. R. China; $^{4}$Three Gorges
Mathematical Research Center, China Three Gorges University, 443002,
Yichang, P. R. China}
\email{tingsongdu@ctgu.edu.cn}
\thanks{$^{\ast }$Corresponding author}

\begin{abstract}
We discover a new version of the celebrated Montgomery identity via quantum
integral operators and establish certain  quantum integral inequalities of
Ostrowski type by using this identity. Relevant connections of the results
obtained in this work with those deduced in earlier published papers are
also considered.
\end{abstract}

\maketitle

\section{Introduction}

The following inequality is named the Ostrowski type inequality \cite{O38}.

\begin{theorem}
\label{1.1}\cite{DR02}. Let $f:\left[ a,b\right] \rightarrow
%TCIMACRO{\U{211d} }%
%BeginExpansion
\mathbb{R}
%EndExpansion
$ be a differentiable mapping on $\left( a,b\right) $ and $f^{\prime }\in L%
\left[ a,b\right] $. If $\left\vert f^{\prime }\left( x\right) \right\vert
<M $ where $x\in \left[ a,b\right] $, then the following inequality holds:%
\begin{equation}
\left\vert f\left( x\right) -\frac{1}{b-a}\int_{a}^{b}f\left( t\right)
dt\right\vert \leq \frac{M}{b-a}\left[ \frac{\left( x-a\right) ^{2}+\left(
b-x\right) ^{2}}{2}\right] \text{ }  \label{1-1}
\end{equation}%
for all $x\in \left[ a,b\right] $.
\end{theorem}

To prove the Ostrowski type inequality above, the following famous
Montgomery identity is very useful, see \cite{MPF91}:%
\begin{equation}
f\left( x\right) =\frac{1}{b-a}\int_{a}^{b}f\left( t\right) dt+\int_{a}^{x}%
\frac{t-a}{b-a}f^{\prime }\left( t\right) dt+\int_{x}^{b}\frac{t-b}{b-a}%
f^{\prime }\left( t\right) dt,  \label{1-2}
\end{equation}%
where $f\left( x\right) $ is a continuous function on $\left[ a,b\right] $
with a continuous first derivative in $\left( a,b\right) $.

By changing variable, the Montgomery identity $\left( \ref{1-2}\right) $ can
be expressed in the following way:

\begin{equation}
f\left( x\right) -\frac{1}{b-a}\int_{a}^{b}f\left( t\right) dt=\left(
b-a\right) \int_{0}^{1}K\left( t\right) f^{\prime }\left( tb+\left(
1-t\right) a\right) dt,  \label{1-3}
\end{equation}%
where
\begin{equation*}
K\left( t\right) =\left\{
\begin{array}{lll}
t, & t\in \big[0,\frac{x-a}{b-a}\big], &  \\
t-1, & t\in \big( \frac{x-a}{b-a},1\big]. &
\end{array}%
\right .
\end{equation*}

A number of different identities of the Montgomery type were investigated
and many Ostrowski type inequalities were obtained by using these
identities. For example, through the framework of Montgomery\textexclamdown
%TCIMACRO{\U{af}}%
%BeginExpansion
\={}%
%EndExpansion
s identity, Cerone and Dragomir \cite{CD03} developed a systematic study
which produced some novel inequalities. By introducing some parameters,
Budak and Sar\i kaya \cite{BS16} as well as \"{O}zdemir et al. \cite{OKA14}
established the generalized Montgomery-type identities for differentiable
mappings and certain generalized Ostrowski-type inequalities, respectively.
Then in \cite{A14}, Aljinovi\'{c} presented another simpler generalization
of the Montgomery identity for fractional integrals by utilizing the
weighted Montgomery identity. Further the generalized Montgomery identity
involving the Ostrowski type inequalities in question with applications to
local fractional integrals can be found in \cite{SB17}. For more related
results considering the different Montgomery identities, The interested
reader is referred, for example, to \cite{A95,KOA11,KS15,L08,SOS12} and the
references therein.

However, to the best of our knowledge, the quantum Montgomery type identity
has not obtained so far. This paper aims to investigate, by setting up a
quantum Montgomery identity and by the help of this identity, some new
quantum integral inequalities such as Ostrowski type, midpoint type, etc. We
shall deal with mappings whose derivatives in absolute value are quantum
differentiable convex mappings.

Throughout this paper, let $0<q<1$ be a constant. It is known that quantum
calculus constructs in a quantum geometric set. That is, if $qx\in A$ for
all $x\in A,$ then the set $A$ is called quantum geometric.

Suppose that $f\left( t\right) $ is an arbitrary function defined on the
interval $\left[ 0,b\right] $. Clearly, for $b>0$, the interval $[0,b]$ is a
quantum geometric set. The quantum derivative of $f\left( t\right) $ is
defined with the following expression:%
\begin{equation}
\begin{array}{c}
D_{q}f\left( t\right)%
\end{array}%
:=\frac{f\left( t\right) -f\left( qt\right) }{\left( 1-q\right) t},t\neq 0,
\label{1-4}
\end{equation}%
\begin{equation*}
\begin{array}{c}
D_{q}f\left( 0\right)%
\end{array}%
:=\underset{t\rightarrow 0}{\lim }%
\begin{array}{c}
D_{q}f\left( t\right).%
\end{array}%
\end{equation*}%
Note that
\begin{equation}
\underset{q\rightarrow 1^{-}}{\lim }%
\begin{array}{c}
D_{q}f\left( t\right)%
\end{array}%
=\frac{df\left( t\right) }{dt},  \label{1-4a}
\end{equation}%
if $f\left( t\right) $ is differentiable.

The definite quantum integral of $f\left( t\right) $ is defined as:%
\begin{equation}
\int_{0}^{b}f\left( t\right) \text{ }d_{q}t=\left( 1-q\right)
b\sum_{n=0}^{\infty }q^{n}f\left( q^{n}b\right)  \label{1-5}
\end{equation}%
and
\begin{equation}
\int_{c}^{b}f\left( t\right) \text{ }d_{q}t=\int_{0}^{b}f\left( t\right)
\text{ }d_{q}t-\int_{0}^{c}f\left( t\right) \text{ }d_{q}t,  \label{1-6}
\end{equation}%
where $0<c<b$, see \cite{AM12,KC01}.

Note that if the series in right-hand side of $\left( \ref{1-5}\right) $ is
convergence, then $\int_{0}^{b}f\left( t\right) $ $d_{q}t$ is exist, i.e., $%
f\left( t\right) $ is quantum integrable on $\left[ 0,b\right] $. Also,
provided that if $\int_{0}^{b}f\left( t\right) $ $dt$ converges, then one
has
\begin{equation}
\underset{q\rightarrow 1^{-}}{\lim }\int_{0}^{b}f\left( t\right) \text{ }%
d_{q}t=\int_{0}^{b}f\left( t\right) \text{ }dt.\text{ \ \ (\cite[page 6]%
{AM12})}  \label{1-6a}
\end{equation}

These definitions are not sufficient in establishing integral inequalities
for a function defined on an arbitrary closed interval $\left[ a,b\right]
\subset
%TCIMACRO{\U{211d} }%
%BeginExpansion
\mathbb{R}
%EndExpansion
$. Due to this fact, Tariboon and Ntouyas in \cite{TN13,TN14} improved these
definitions as follows:

\begin{definition}
\label{1.2} \cite{TN14,TN13}. For a continuous function $f:\left[ a,b\right]
\rightarrow
%TCIMACRO{\U{211d} }%
%BeginExpansion
\mathbb{R}
%EndExpansion
$, the $q$-derivative of $f$ at $t\in \left[ a,b\right] $ is characterized
by the expression:
\begin{equation}
\begin{array}{c}
_{a}D_{q}f\left( t\right)%
\end{array}%
=\frac{f\left( t\right) -f\left( qt+\left( 1-q\right) a\right) }{\left(
1-q\right) \left( t-a\right) },\text{ }t\neq a,  \label{1-7}
\end{equation}%
\begin{equation*}
\begin{array}{c}
_{a}D_{q}f\left( a\right)%
\end{array}%
=\underset{t\rightarrow a}{\lim }%
\begin{array}{c}
_{a}D_{q}f\left( t\right).%
\end{array}%
\end{equation*}%
The function $f$ is said to be $q$-differentiable on $\left[ a,b\right] $,
if $%
\begin{array}{c}
_{a}D_{q}f\left( t\right)%
\end{array}%
$ exists for all $t\in \left[ a,b\right] $.
\end{definition}

Clearly, if $a=0$ in $\left( \ref{1-7}\right) $, then$%
\begin{array}{c}
_{0}D_{q}f\left( t\right) =D_{q}f\left( t\right),
\end{array}%
$ where $D_{q}f\left( t\right) $ is familiar quantum derivatives given in $%
\left( \ref{1-4}\right) $.

\begin{definition}
\label{1.3} \cite{TN14,TN13}. Let $f:\left[ a,b\right] \rightarrow
%TCIMACRO{\U{211d} }%
%BeginExpansion
\mathbb{R}
%EndExpansion
$ be a continuous function. Then the quantum definite integral on $\left[ a,b%
\right] $ is delineated as%
\begin{equation}
\int_{a}^{b}f\left( t\right)
\begin{array}{c}
_{a}d_{q}t%
\end{array}%
=\left( 1-q\right) \left( b-a\right) \sum\limits_{n=0}^{\infty }q^{n}f\left(
q^{n}b+\left( 1-q^{n}\right) a\right)  \label{1-8}
\end{equation}%
and%
\begin{equation}
\int_{c}^{b}f\left( t\right)
\begin{array}{c}
_{a}d_{q}t%
\end{array}%
=\int_{a}^{b}f\left( t\right)
\begin{array}{c}
_{a}d_{q}t%
\end{array}%
-\int_{a}^{c}f\left( t\right)
\begin{array}{c}
_{a}d_{q}t,%
\end{array}
\label{1-9}
\end{equation}%
where $a<c<b$.
\end{definition}

Clearly, if $a=0$ in $\left( \ref{1-8}\right) $, then
\begin{equation*}
\int_{0}^{b}f\left( t\right)
\begin{array}{c}
_{0}d_{q}t%
\end{array}%
=\int_{0}^{b}f\left( t\right)
\begin{array}{c}
d_{q}t,%
\end{array}%
\end{equation*}%
where $\int_{0}^{b}f\left( t\right)
\begin{array}{c}
d_{q}t%
\end{array}%
$ is familiar definite quantum integrals on $\left[ 0,b\right] $ given in $%
\left( \ref{1-5}\right) $.

Definition \ref{1.2} and Definition \ref{1.3} have actually developed
previous definitions and have been widely used for quantum integral
inequalities. There is a lot of remarkable papers about quantum integral
inequalities based on these definitions, including Kunt et al. \cite{KLID19}
in the study of the quantum Hermite--Hadamard inequalities for mappings of
two variables considering convexity and quasi-convexity on the co-ordinates,
Noor et al. \cite{NAN16,NNA15,NNA15a} in quantum Ostrowski-type inequalities
for quantum differentiable convex mappings, quantum estimates for
Hermite--Hadamard inequalities via convexity and quasi-convexity, quantum
analogues of Iyengar type inequalities for some classes of preinvex
mappings, as well as Tun\c{c} et al. \cite{TGB18} in the Simpson-type
inequalities for convex mappings via quantum integral operators. For more
results related to the quantum integral operators, the interested reader is
directed, for example, to \cite{AS17,KIAS18, Liu2017, SNT15,ZLP19} and the
references cited therein.

In \cite{ASKI16}, Alp et al. proved the following inequality named quantum
Hermite--Hadamard type inequality. Also in \cite{ZDWS18}, Zhang et al.
proved the same inequality with the fewer assumptions and shorter method.

\begin{theorem}
\label{1.4}Let $f:\left[ a,b\right] \rightarrow
%TCIMACRO{\U{211d} }%
%BeginExpansion
\mathbb{R}
%EndExpansion
$ be a convex function with $0<q<1$. Then we have%
\begin{equation}
f\left( \frac{qa+b}{1+q}\right) \leq \frac{1}{b-a}\int_{a}^{b}f\left(
t\right)
\begin{array}{c}
_{a}d_{q}t%
\end{array}%
\leq \frac{qf\left( a\right) +f\left( b\right) }{1+q}.  \label{1-10}
\end{equation}
\end{theorem}

\section{Main Results}

Firstly, we shall discuss the assumptions of being continuous of the
function $f\left( t\right) $ in the Definition \ref{1.2} and the Definition %
\ref{1.3}. Also, in this conditions, we want to discuss if the similar cases
with $\left( \ref{1-4a}\right) $ and $\left( \ref{1-6a}\right) $ could be
exist.

Considering the Definition \ref{1.2}, it is unnecessary that the function $%
f\left( t\right) $ is being continuous on $\left[ a,b\right] $. Indeed, for
all $t\in \left[ a,b\right] $, $qt+\left( 1-q\right) a\in \left[ a,b\right] $
and $f\left( t\right) -f\left( qt+\left( 1-q\right) a\right) \in
%TCIMACRO{\U{211d} }%
%BeginExpansion
\mathbb{R}
%EndExpansion
$. It means that $\frac{f\left( t\right) -f\left( qt+\left( 1-q\right)
a\right) }{\left( 1-q\right) \left( t-a\right) }\in
%TCIMACRO{\U{211d} }%
%BeginExpansion
\mathbb{R}
%EndExpansion
$ exists for all $t\neq a$, i.e., the Definition \ref{1.2} should be as
follows:

\begin{definition}
\label{2.1}(Quantum derivative on $\left[ a,b\right] $) An arbitrary
function $f\left( t\right) $ defined on $\left[ a,b\right] $ is called
quantum differentiable on $\left( a,b\right] $ with the following expression:%
\begin{equation}
\begin{array}{c}
_{a}D_{q}f\left( t\right)%
\end{array}%
=\frac{f\left( t\right) -f\left( qt+\left( 1-q\right) a\right) }{\left(
1-q\right) \left( t-a\right) }\in
%TCIMACRO{\U{211d} }%
%BeginExpansion
\mathbb{R}
%EndExpansion
,\text{ }t\neq a  \label{2-1a}
\end{equation}%
and quantum differentiable on $t=a$, if the following limit exists:%
\begin{equation*}
\begin{array}{c}
_{a}D_{q}f\left( a\right)%
\end{array}%
=\underset{t\rightarrow a}{\lim }%
\begin{array}{c}
_{a}D_{q}f\left( t\right).%
\end{array}%
\end{equation*}
\end{definition}

\begin{lemma}
\label{2.2}(Similar case with $\left( \ref{1-4a}\right) $) Let $f:\left[ a,b%
\right] \rightarrow
%TCIMACRO{\U{211d} }%
%BeginExpansion
\mathbb{R}
%EndExpansion
$ be a differentiable function. Then we have%
\begin{equation}
\underset{q\rightarrow 1^{-}}{\lim }%
\begin{array}{c}
_{a}D_{q}f\left( t\right)%
\end{array}%
=\frac{df\left( t\right) }{dt}.  \label{2-1}
\end{equation}
\end{lemma}

\begin{proof}
If $f\left( t\right) $ is differentiable on $\left[ a,b\right] $, then we
have that%
\begin{equation*}
\underset{q\rightarrow 1^{-}}{\lim }%
\begin{array}{c}
_{a}D_{q}f\left( t\right)%
\end{array}%
=\underset{q\rightarrow 1^{-}}{\lim }\frac{f\left( t\right) -f\left(
qt+\left( 1-q\right) a\right) }{\left( 1-q\right) \left( t-a\right) }
\end{equation*}%
\begin{equation*}
\overset{\frac{0}{0}}{=}\underset{q\rightarrow 1^{-}}{\lim }\frac{-\left(
t-a\right) \frac{df\left( qt+\left( 1-q\right) a\right) }{dq}}{-\left(
t-a\right) }
\end{equation*}%
\begin{equation*}
=\underset{q\rightarrow 1^{-}}{\lim }\frac{df\left( qt+\left( 1-q\right)
a\right) }{dq}
\end{equation*}%
\begin{equation*}
=\underset{q\rightarrow 1^{-}}{\lim }\left( \underset{h\rightarrow 0}{\lim }%
\frac{f\left( \left( q+h\right) t+\left( 1-\left( q+h\right) \right)
a\right) -f\left( qt+\left( 1-q\right) a\right) }{h}\right)
\end{equation*}%
\begin{equation*}
=\underset{h\rightarrow 0}{\lim }\left( \underset{q\rightarrow 1^{-}}{\lim }%
\frac{f\left( \left( q+h\right) t+\left( 1-\left( q+h\right) \right)
a\right) -f\left( qt+\left( 1-q\right) a\right) }{h}\right)
\end{equation*}%
\begin{equation*}
=\underset{h\rightarrow 0}{\lim }\frac{f\left( t+ht-ha\right) -f\left(
t\right) }{h}
\end{equation*}%
\begin{equation*}
=\underset{h\rightarrow 0}{\lim }\frac{f\left( t+h\right) -f\left( t\right)
}{h}
\end{equation*}
\begin{equation*}
=\frac{df\left( t\right) }{dt}.
\end{equation*}
\end{proof}

Considering Definition \ref{1.3}, it is unnecessary that the function $%
f\left( t\right) $ is being continuous on $\left[ a,b\right] $. Undoubtedly,
it is not difficult to construct an example for a discontinuous function
that is quantum integrable on $\left[ a,b\right] $.

\begin{example}
\label{2.3}Let $0<q<1$ be a constant, and we define\newline
$A:=\left\{ q^{n}2+\left( 1-q^{n}\right) \left( -1\right)
:n=0,1,2,...,\right\} \subset \left[ -1,2\right] $, $f:\left[ -1,2\right]
\rightarrow
%TCIMACRO{\U{211d} }%
%BeginExpansion
\mathbb{R}
%EndExpansion
$and $f\left( t\right) :=\left\{
\begin{array}{ccc}
1, & t\in A, &  \\
0, \ \  & t\in \left[ -1,2\right] \backslash A. &
\end{array}%
\right. $ Clearly, the function $f\left( t\right) $ is not continuous on $%
\left[ -1,2\right] $. On the other hand%
\begin{equation*}
\int_{-1}^{2}f\left( t\right)
\begin{array}{c}
_{-1}d_{q}t%
\end{array}%
=\left( 1-q\right) \left( 2-\left( -1\right) \right) \sum_{n=0}^{\infty
}q^{n}f\left( q^{n}2+\left( 1-q^{n}\right) \left( -1\right) \right)
\end{equation*}%
\begin{equation*}
=3\left( 1-q\right) \sum_{n=0}^{\infty }q^{n}=3\left( 1-q\right) \frac{1}{1-q%
}=3,
\end{equation*}
i.e., the function $f\left( t\right) $ is quantum integrable on $\left[ -1,2%
\right]. $
\end{example}

Similarly, the Definition \ref{1.3} should be described in the following way.

\begin{definition}
\label{2.4}(Quantum definite integral on $\left[ a,b\right] $) Let $f:\left[
a,b\right] \rightarrow
%TCIMACRO{\U{211d} }%
%BeginExpansion
\mathbb{R}
%EndExpansion
$ be an arbitrary function. Then the quantum definite integral on $\left[ a,b%
\right] $ is delineated as%
\begin{equation}
\int_{a}^{b}f\left( t\right)
\begin{array}{c}
_{a}d_{q}t%
\end{array}%
=\left( 1-q\right) \left( b-a\right) \sum\limits_{n=0}^{\infty }q^{n}f\left(
q^{n}b+\left( 1-q^{n}\right) a\right).  \label{2-2}
\end{equation}
If the series in right hand-side of $\left( \ref{2-2}\right) $ is
convergent, then $\int_{a}^{b}f\left( t\right)
\begin{array}{c}
_{a}d_{q}t%
\end{array}%
$ is exist, i.e., $f\left( t\right) $ is quantum integrable on $\left[ a,b%
\right] $.
\end{definition}

\begin{lemma}
\label{2.5}(Similar case with $\left( \ref{1-6a}\right) $) Let $f:\left[ a,b%
\right] \rightarrow
%TCIMACRO{\U{211d} }%
%BeginExpansion
\mathbb{R}
%EndExpansion
$ be an arbitrary function. Then, provided that if $\int_{a}^{b}f\left(
t\right) $ $dt$ converges, then we have
\begin{equation}
\underset{q\rightarrow 1^{-}}{\lim }\int_{a}^{b}f\left( t\right)
\begin{array}{c}
_{a}d_{q}t%
\end{array}%
=\int_{a}^{b}f\left( t\right) \text{ }dt.  \label{2-3}
\end{equation}
\end{lemma}

\begin{proof}
If $\int_{a}^{b}f\left( t\right) $ $dt$ converges, then $\int_{0}^{1}f\left(
tb+\left( 1-t\right) a\right) $ $dt$ also converges. Using $\left( \ref{1-6a}%
\right) $, we have that%
\begin{equation*}
\underset{q\rightarrow 1^{-}}{\lim }\int_{a}^{b}f\left( t\right)
\begin{array}{c}
_{a}d_{q}t%
\end{array}%
=\underset{q\rightarrow 1^{-}}{\lim }\left[ \left( 1-q\right) \left(
b-a\right) \sum\limits_{n=0}^{\infty }q^{n}f\left( q^{n}b+\left(
1-q^{n}\right) a\right) \right]
\end{equation*}%
\begin{equation*}
=\left( b-a\right) \underset{q\rightarrow 1^{-}}{\lim }\int_{0}^{1}f\left(
tb+\left( 1-t\right) a\right)
\begin{array}{c}
_{0}d_{q}t%
\end{array}%
\end{equation*}%
\begin{equation*}
=\left( b-a\right) \int_{0}^{1}f\left( tb+\left( 1-t\right) a\right) dt
\end{equation*}%
\begin{equation*}
=\int_{a}^{b}f\left( t\right) \text{ }dt.
\end{equation*}
\end{proof}

We next present a important quantum Montgomery identity, which is similar
with the identity in $\left( \ref{1-3}\right) $.

\begin{lemma}
\label{2.6}(Quantum Montgomery identity) Let $f:\left[ a,b\right]
\rightarrow
%TCIMACRO{\U{211d} }%
%BeginExpansion
\mathbb{R}
%EndExpansion
$ be an arbitrary function with $%
\begin{array}{c}
_{a}D_{q}f%
\end{array}%
$ is quantum integrable on $\left[ a,b\right] $, then the following quantum
identity holds:%
\begin{equation}
f\left( x\right) -\frac{1}{b-a}\int_{a}^{b}f\left( t\right)
\begin{array}{c}
_{a}d_{q}t%
\end{array}%
=\left( b-a\right) \int_{0}^{1}K_{q}\left( t\right)
\begin{array}{c}
_{a}D_{q}f\left( tb+\left( 1-t\right) a\right)%
\end{array}%
\begin{array}{c}
_{0}d_{q}t,%
\end{array}
\label{2-4}
\end{equation}%
where%
\begin{equation*}
K_{q}\left( t\right) =\left\{
\begin{array}{lll}
qt & , & t\in \big[ 0,\frac{x-a}{b-a}\big], \\
qt-1 & , & t\in \big( \frac{x-a}{b-a},1\big].%
\end{array}%
\right.
\end{equation*}
\end{lemma}

\begin{proof}
By the Definition \ref{2.1}, $f\left( t\right) $ is quantum differentiable
on $\left( a,b\right) $ and $%
\begin{array}{c}
_{a}D_{q}f%
\end{array}%
$ is exist. Since $%
\begin{array}{c}
_{a}D_{q}f%
\end{array}%
$ is quantum integrable on $\left[ a,b\right] $, by the Definition \ref{2.4}%
, the quantum integral for the right-side of $\left( \ref{2-4}\right) $ is
exist. Let us start calculating the integral for the right-side of $\left( %
\ref{2-4}\right). $ With the help of $\left( \ref{2-1a}\right) $ and $\left( %
\ref{2-2}\right) $, we have that%
\begin{equation*}
\left( b-a\right) \int_{0}^{1}K_{q}\left( t\right)
\begin{array}{c}
_{a}D_{q}f\left( tb+\left( 1-t\right) a\right)%
\end{array}%
\begin{array}{c}
_{0}d_{q}t%
\end{array}%
\end{equation*}%
\begin{equation*}
=\left( b-a\right) \left[
\begin{array}{l}
\int_{0}^{\frac{x-a}{b-a}}qt%
\begin{array}{c}
_{a}D_{q}f\left( tb+\left( 1-t\right) a\right)%
\end{array}%
\begin{array}{c}
_{0}d_{q}t%
\end{array}
\\
+\int_{\frac{x-a}{b-a}}^{1}\left( qt-1\right)
\begin{array}{c}
_{a}D_{q}f\left( tb+\left( 1-t\right) a\right)%
\end{array}%
\begin{array}{c}
_{0}d_{q}t%
\end{array}%
\end{array}%
\right]
\end{equation*}%
\begin{equation*}
=\left( b-a\right) \left[
\begin{array}{l}
\int_{0}^{\frac{x-a}{b-a}}qt%
\begin{array}{c}
_{a}D_{q}f\left( tb+\left( 1-t\right) a\right)%
\end{array}%
\begin{array}{c}
_{0}d_{q}t%
\end{array}
\\
+\int_{0}^{1}\left( qt-1\right)
\begin{array}{c}
_{a}D_{q}f\left( tb+\left( 1-t\right) a\right)%
\end{array}%
\begin{array}{c}
_{0}d_{q}t%
\end{array}
\\
-\int_{0}^{\frac{x-a}{b-a}}\left( qt-1\right)
\begin{array}{c}
_{a}D_{q}f\left( tb+\left( 1-t\right) a\right)%
\end{array}%
\begin{array}{c}
_{0}d_{q}t%
\end{array}%
\end{array}%
\right]
\end{equation*}%
\begin{equation*}
=\left( b-a\right) \left[
\begin{array}{l}
\int_{0}^{1}\left( qt-1\right)
\begin{array}{c}
_{a}D_{q}f\left( tb+\left( 1-t\right) a\right)%
\end{array}%
\begin{array}{c}
_{0}d_{q}t%
\end{array}
\\
+\int_{0}^{\frac{x-a}{b-a}}%
\begin{array}{c}
_{a}D_{q}f\left( tb+\left( 1-t\right) a\right)%
\end{array}%
\begin{array}{c}
_{0}d_{q}t%
\end{array}%
\end{array}%
\right]
\end{equation*}%
\begin{equation*}
=\left( b-a\right) \left[
\begin{array}{l}
\int_{0}^{1}qt%
\begin{array}{c}
_{a}D_{q}f\left( tb+\left( 1-t\right) a\right)%
\end{array}%
\begin{array}{c}
_{0}d_{q}t%
\end{array}
\\
-\int_{0}^{1}%
\begin{array}{c}
_{a}D_{q}f\left( tb+\left( 1-t\right) a\right)%
\end{array}%
\begin{array}{c}
_{0}d_{q}t%
\end{array}
\\
+\int_{0}^{\frac{x-a}{b-a}}%
\begin{array}{c}
_{a}D_{q}f\left( tb+\left( 1-t\right) a\right)%
\end{array}%
\begin{array}{c}
_{0}d_{q}t%
\end{array}%
\end{array}%
\right]
\end{equation*}%
\begin{equation*}
=\left( b-a\right) \left[
\begin{array}{l}
\int_{0}^{1}qt\frac{f\left( tb+\left( 1-t\right) a\right) -f\left(
qtb+\left( 1-qt\right) a\right) }{\left( 1-q\right) t\left( b-a\right) }%
\begin{array}{c}
_{0}d_{q}t%
\end{array}
\\
-\int_{0}^{1}\frac{f\left( tb+\left( 1-t\right) a\right) -f\left( qtb+\left(
1-qt\right) a\right) }{\left( 1-q\right) t\left( b-a\right) }%
\begin{array}{c}
_{0}d_{q}t%
\end{array}
\\
+\int_{0}^{\frac{x-a}{b-a}}%
\begin{array}{c}
\frac{f\left( tb+\left( 1-t\right) a\right) -f\left( qtb+\left( 1-qt\right)
a\right) }{\left( 1-q\right) t\left( b-a\right) }%
\end{array}%
\begin{array}{c}
_{0}d_{q}t%
\end{array}%
\end{array}%
\right]
\end{equation*}%
\begin{equation*}
=\frac{1}{1-q}\left[
\begin{array}{l}
q\left[ \int_{0}^{1}f\left( tb+\left( 1-t\right) a\right)
\begin{array}{c}
_{0}d_{q}t%
\end{array}%
-\int_{0}^{1}f\left( qtb+\left( 1-qt\right) a\right)
\begin{array}{c}
_{0}d_{q}t%
\end{array}%
\right] \\
-\left[ \int_{0}^{1}\frac{f\left( tb+\left( 1-t\right) a\right) }{t}%
\begin{array}{c}
_{0}d_{q}t%
\end{array}%
-\int_{0}^{1}\frac{f\left( qtb+\left( 1-qt\right) a\right) }{t}%
\begin{array}{c}
_{0}d_{q}t%
\end{array}%
\right] \\
+\left[ \int_{0}^{\frac{x-a}{b-a}}%
\begin{array}{c}
\frac{f\left( tb+\left( 1-t\right) a\right) }{t}%
\end{array}%
\begin{array}{c}
_{0}d_{q}t%
\end{array}%
-\int_{0}^{\frac{x-a}{b-a}}%
\begin{array}{c}
\frac{f\left( qtb+\left( 1-qt\right) a\right) }{t}%
\end{array}%
\begin{array}{c}
_{0}d_{q}t%
\end{array}%
\right]%
\end{array}%
\right]
\end{equation*}%
\begin{equation*}
=\frac{1}{1-q}\left[
\begin{array}{l}
q\left[
\begin{array}{c}
\left( 1-q\right) \sum\limits_{n=0}^{\infty }q^{n}f\left( q^{n}b+\left(
1-q^{n}\right) a\right) \\
-\left( 1-q\right) \sum\limits_{n=0}^{\infty }q^{n}f\left( q^{n+1}b+\left(
1-q^{n+1}\right) a\right)%
\end{array}%
\right] \\
-\left[
\begin{array}{c}
\left( 1-q\right) \sum\limits_{n=0}^{\infty }q^{n}\frac{f\left(
q^{n}b+\left( 1-q^{n}\right) a\right) }{q^{n}} \\
-\left( 1-q\right) \sum\limits_{n=0}^{\infty }q^{n}\frac{f\left(
q^{n+1}b+\left( 1-q^{n+1}\right) a\right) }{q^{n}}%
\end{array}%
\right] \\
+\left[
\begin{array}{c}
\left( 1-q\right) \frac{x-a}{b-a}\sum\limits_{n=0}^{\infty }q^{n}\frac{%
f\left( q^{n}\frac{x-a}{b-a}b+\left( 1-q^{n}\frac{x-a}{b-a}\right) a\right)
}{q^{n}\frac{x-a}{b-a}} \\
-\left( 1-q\right) \frac{x-a}{b-a}\sum\limits_{n=0}^{\infty }q^{n}\frac{%
f\left( q^{n+1}\frac{x-a}{b-a}b+\left( 1-q^{n+1}\frac{x-a}{b-a}\right)
a\right) }{q^{n}\frac{x-a}{b-a}}%
\end{array}%
\right]%
\end{array}%
\right]
\end{equation*}%
\begin{equation*}
=\left[
\begin{array}{l}
q\left[ \sum\limits_{n=0}^{\infty }q^{n}f\left( q^{n}b+\left( 1-q^{n}\right)
a\right) -\sum\limits_{n=0}^{\infty }q^{n}f\left( q^{n+1}b+\left(
1-q^{n+1}\right) a\right) \right] \\
-\left[ \sum\limits_{n=0}^{\infty }f\left( q^{n}b+\left( 1-q^{n}\right)
a\right) -\sum\limits_{n=0}^{\infty }f\left( q^{n+1}b+\left(
1-q^{n+1}\right) a\right) \right] \\
+\left[
\begin{array}{c}
\sum\limits_{n=0}^{\infty }f\left( q^{n}\frac{x-a}{b-q}b+\left( 1-q^{n}\frac{%
x-a}{b-q}\right) a\right) \\
-\sum\limits_{n=0}^{\infty }f\left( q^{n+1}\frac{x-a}{b-q}b+\left( 1-q^{n+1}%
\frac{x-a}{b-q}\right) a\right)%
\end{array}%
\right]%
\end{array}%
\right]
\end{equation*}%
\begin{equation*}
=\left[
\begin{array}{l}
q\left[ \sum\limits_{n=0}^{\infty }q^{n}f\left( q^{n}b+\left( 1-q^{n}\right)
a\right) -\frac{1}{q}\sum\limits_{n=1}^{\infty }q^{n}f\left( q^{n}b+\left(
1-q^{n}\right) a\right) \right] \\
-\left[ \sum\limits_{n=0}^{\infty }f\left( q^{n}b+\left( 1-q^{n}\right)
a\right) -\sum\limits_{n=1}^{\infty }f\left( q^{n}b+\left( 1-q^{n}\right)
a\right) \right] \\
+\left[
\begin{array}{c}
\sum\limits_{n=0}^{\infty }f\left( q^{n}\left( \frac{x-a}{b-a}\right)
b+\left( 1-q^{n}\left( \frac{x-a}{b-a}\right) \right) a\right) \\
-\sum\limits_{n=1}^{\infty }f\left( q^{n}\left( \frac{x-a}{b-a}\right)
b+\left( 1-q^{n}\left( \frac{x-a}{b-a}\right) \right) a\right)%
\end{array}%
\right]%
\end{array}%
\right]
\end{equation*}%
\begin{equation*}
=\left[
\begin{array}{l}
q\left[ \left( 1-\frac{1}{q}\right) \sum\limits_{n=0}^{\infty }q^{n}f\left(
q^{n}b+\left( 1-q^{n}\right) a\right) +\frac{f\left( b\right) }{q}\right] \\
-f\left( b\right) +f\left( \left( \frac{x-a}{b-a}\right) b+\left( 1-\left(
\frac{x-a}{b-a}\right) \right) a\right)%
\end{array}%
\right]
\end{equation*}%
\begin{equation*}
=f\left( x\right) -\left( 1-q\right) \sum\limits_{n=0}^{\infty }q^{n}f\left(
q^{n}b+\left( 1-q^{n}\right) a\right)
\end{equation*}%
\begin{equation*}
=f\left( x\right) -\frac{1}{b-a}\int_{a}^{b}f\left( t\right)
\begin{array}{c}
_{a}d_{q}t,%
\end{array}%
\end{equation*}
which completes the proof.
\end{proof}

\begin{remark}
If one takes limit $q\rightarrow 1^{-}$ on the Quantum Montgomery identity
in $\left( \ref{2-4}\right) $, one has the Montgomery identity in $\left( %
\ref{1-3}\right) $.
\end{remark}

The following calculations of quantum definite integrals are used in next
result:%
\begin{eqnarray}
\int_{0}^{\frac{x-a}{b-a}}qt%
\begin{array}{c}
_{0}d_{q}t%
\end{array}
&=&q\left( 1-q\right) \frac{x-a}{b-a}\sum_{n=0}^{\infty }q^{n}\left( \frac{%
x-a}{b-a}q^{n}\right)   \label{2-5} \\
&=&q\left( 1-q\right) \left( \frac{x-a}{b-a}\right) ^{2}\frac{1}{1-q^{2}}
\notag \\
&=&\frac{q}{1+q}\left( \frac{x-a}{b-a}\right) ^{2},  \notag
\end{eqnarray}%
\begin{eqnarray}
\int_{0}^{\frac{x-a}{b-a}}qt^{2}%
\begin{array}{c}
_{0}d_{q}t%
\end{array}
&=&q\left( 1-q\right) \frac{x-a}{b-a}\sum_{n=0}^{\infty }q^{n}\left( \frac{%
x-a}{b-a}q^{n}\right) ^{2}  \label{2-6} \\
&=&q\left( 1-q\right) \left( \frac{x-a}{b-a}\right) ^{3}\frac{1}{1-q^{3}}
\notag \\
&=&\frac{q}{1+q+q^{2}}\left( \frac{x-a}{b-a}\right) ^{3},  \notag
\end{eqnarray}%
$\allowbreak $%
\begin{eqnarray}
&&\int_{\frac{x-a}{b-a}}^{1}\left( 1-qt\right)
\begin{array}{c}
_{0}d_{q}t%
\end{array}
\label{2-7} \\
&=&\int_{0}^{1}\left( 1-qt\right)
\begin{array}{c}
_{0}d_{q}t%
\end{array}%
-\int_{0}^{\frac{x-a}{b-a}}\left( 1-qt\right)
\begin{array}{c}
_{0}d_{q}t%
\end{array}
\notag \\
&=&\left[
\begin{array}{c}
\left( 1-q\right) \sum_{n=0}^{\infty }q^{n}\left( 1-qq^{n}\right)  \\
-\left( 1-q\right) \frac{x-a}{b-a}\sum_{n=0}^{\infty }q^{n}\left( 1-qq^{n}%
\frac{x-a}{b-a}\right)
\end{array}%
\right]   \notag \\
&=&\left[
\begin{array}{c}
\left( 1-q\right) \left( \frac{1}{1-q}-\frac{q}{1-q^{2}}\right)  \\
-\left( 1-q\right) \frac{x-a}{b-a}\left( \frac{1}{1-q}-\frac{q}{1-q^{2}}%
\frac{x-a}{b-a}\right)
\end{array}%
\right]   \notag \\
&=&\frac{1}{1+q}-\frac{x-a}{b-a}\left( 1-\frac{q}{1+q}\frac{x-a}{b-a}\right)
\notag \\
&=&\frac{1}{1+q}-\left( 1-\frac{b-x}{b-a}\right) \left( \frac{1}{1+q}+\frac{q%
}{1+q}-\frac{q}{1+q}\left( 1-\frac{b-x}{b-a}\right) \right)   \notag \\
&=&\frac{1}{1+q}-\left( 1-\frac{b-x}{b-a}\right) \left( \frac{1}{1+q}+\frac{q%
}{1+q}\left( \frac{b-x}{b-a}\right) \right)   \notag \\
&=&\left[
\begin{array}{c}
\frac{1}{1+q}-\frac{1}{1+q}-\frac{q}{1+q}\left( \frac{b-x}{b-a}\right)  \\
+\frac{q}{1+q}\left( \frac{b-x}{b-a}\right) +\frac{q}{1+q}\left( \frac{b-x}{%
b-a}\right) ^{2}%
\end{array}%
\right] =\frac{q}{1+q}\left( \frac{b-x}{b-a}\right) ^{2},  \notag
\end{eqnarray}%
and
\begin{eqnarray}
&&\int_{\frac{x-a}{b-a}}^{1}\left( t-qt^{2}\right)
\begin{array}{c}
_{0}d_{q}t%
\end{array}
\label{2-8} \\
&=&\int_{0}^{1}\left( t-qt^{2}\right)
\begin{array}{c}
_{0}d_{q}t%
\end{array}%
-\int_{0}^{\frac{x-a}{b-a}}\left( t-qt^{2}\right)
\begin{array}{c}
_{0}d_{q}t%
\end{array}
\notag \\
&=&\left[
\begin{array}{c}
\left( 1-q\right) \sum_{n=0}^{\infty }q^{n}\left( q^{n}-qq^{2n}\right)  \\
-\left( 1-q\right) \frac{x-a}{b-a}\sum_{n=0}^{\infty }q^{n}\left( q^{n}\frac{%
x-a}{b-a}-qq^{2n}\left( \frac{x-a}{b-a}\right) ^{2}\right)
\end{array}%
\right]   \notag \\
&=&\left[
\begin{array}{c}
\left( 1-q\right) \left( \frac{1}{1-q^{2}}-\frac{q}{1-q^{3}}\right)  \\
-\left( 1-q\right) \frac{x-a}{b-a}\left( \frac{1}{1-q^{2}}\frac{x-a}{b-a}-%
\frac{q}{1-q^{3}}\left( \frac{x-a}{b-a}\right) ^{2}\right)
\end{array}%
\right]   \notag \\
&=&\left[
\begin{array}{c}
\left( \frac{1}{1+q}-\frac{q}{1+q+q^{2}}\right)  \\
-\frac{x-a}{b-a}\left( \frac{1}{1+q}\frac{x-a}{b-a}-\frac{q}{1+q+q^{2}}%
\left( \frac{x-a}{b-a}\right) ^{2}\right)
\end{array}%
\right]   \notag \\
&=&\left[
\begin{array}{c}
\frac{1}{\left( 1+q\right) \left( 1+q+q^{2}\right) }-\frac{1}{1+q}\left(
\frac{x-a}{b-a}\right) ^{2} \\
+\frac{q}{1+q+q^{2}}\left( \frac{x-a}{b-a}\right) ^{3}%
\end{array}%
\right] .  \notag
\end{eqnarray}

Let us introduce some new quantum integral inequalities by the help of
quantum power mean inequality and Lemma \ref{2.6}.

\begin{theorem}
\label{2.8} Let $f:\left[ a,b\right] \rightarrow
%TCIMACRO{\U{211d} }%
%BeginExpansion
\mathbb{R}
%EndExpansion
$ be an arbitrary function with $%
\begin{array}{c}
_{a}D_{q}f%
\end{array}%
$ is quantum integrable on $\left[ a,b\right] $. If $%
\begin{array}{c}
\left\vert _{a}D_{q}f\right\vert ^{r}%
\end{array}%
$, $r\geq 1$ is a convex function, then the following quantum integral
inequality holds:%
\begin{eqnarray}
&&\left\vert f\left( x\right) -\frac{1}{b-a}\int_{a}^{b}f\left( t\right)
\begin{array}{c}
_{a}d_{q}t%
\end{array}%
\right\vert   \label{2-9} \\
&\leq &\left( b-a\right) \left[
\begin{array}{c}
K_{1}^{1-\frac{1}{r}}\left( a,b,x,q\right) \left[
\begin{array}{c}
\left\vert _{a}D_{q}f\left( a\right) \right\vert ^{r}K_{2}\left(
a,b,x,q\right)  \\
+\left\vert _{a}D_{q}f\left( b\right) \right\vert ^{r}K_{3}\left(
a,b,x,q\right)
\end{array}%
\right] ^{\frac{1}{r}} \\
+K_{4}^{1-\frac{1}{r}}\left( a,b,x,q\right) \left[
\begin{array}{c}
\left\vert _{a}D_{q}f\left( a\right) \right\vert ^{r}K_{5}\left(
a,b,x,q\right)  \\
+\left\vert _{a}D_{q}f\left( b\right) \right\vert ^{r}K_{6}\left(
a,b,x,q\right)
\end{array}%
\right] ^{\frac{1}{r}}%
\end{array}%
\right]   \notag
\end{eqnarray}%
for all $x\in \left[ a,b\right] $, where%
\begin{equation*}
K_{1}\left( a,b,x,q\right) =\int_{0}^{\frac{x-a}{b-a}}qt%
\begin{array}{c}
_{0}d_{q}t%
\end{array}%
=\frac{q}{1+q}\left( \frac{x-a}{b-a}\right) ^{2},
\end{equation*}%
\begin{equation*}
K_{2}\left( a,b,x,q\right) =\int_{0}^{\frac{x-a}{b-a}}qt^{2}%
\begin{array}{c}
_{0}d_{q}t%
\end{array}%
=\frac{q}{1+q+q^{2}}\left( \frac{x-a}{b-a}\right) ^{3},
\end{equation*}%
\begin{equation*}
K_{3}\left( a,b,x,q\right) =\int_{0}^{\frac{x-a}{b-a}}qt-qt^{2}%
\begin{array}{c}
_{0}d_{q}t%
\end{array}%
=K_{1}\left( a,b,x,q\right) -K_{2}\left( a,b,x,q\right) ,
\end{equation*}%
$\allowbreak $%
\begin{equation*}
K_{4}\left( a,b,x,q\right) =\int_{\frac{x-a}{b-a}}^{1}\left( 1-qt\right)
\begin{array}{c}
_{0}d_{q}t%
\end{array}%
=\frac{q}{1+q}\left( \frac{b-x}{b-a}\right) ^{2},
\end{equation*}%
\begin{equation*}
K_{5}\left( a,b,x,q\right) =\int_{\frac{x-a}{b-a}}^{1}\left( t-qt^{2}\right)
\begin{array}{c}
_{0}d_{q}t%
\end{array}%
=\left[
\begin{array}{c}
\frac{1}{\left( 1+q\right) \left( 1+q+q^{2}\right) }-\frac{1}{1+q}\left(
\frac{x-a}{b-a}\right) ^{2} \\
+\frac{q}{1+q+q^{2}}\left( \frac{x-a}{b-a}\right) ^{3}%
\end{array}%
\right] ,
\end{equation*}%
and
\begin{equation*}
K_{6}\left( a,b,x,q\right) =\int_{\frac{x-a}{b-a}}^{1}\left(
1-qt-t+qt^{2}\right)
\begin{array}{c}
_{0}d_{q}t%
\end{array}%
=K_{4}\left( a,b,x,q\right) -K_{5}\left( a,b,x,q\right) .
\end{equation*}
\end{theorem}

\begin{proof}
Using convexity of $\left\vert _{a}D_{q}f\right\vert ^{r}$, we have that
\begin{equation}
\left\vert _{a}D_{q}f\left( tb+\left( 1-t\right) a\right) \right\vert
^{r}\leq t\left\vert _{a}D_{q}f\left( a\right) \right\vert ^{r}+\left(
1-t\right) \left\vert _{a}D_{q}f\left( b\right) \right\vert ^{r}.
\label{2-10}
\end{equation}
By using Lemma \ref{2.6}, quantum power mean inequality and $\left( \ref%
{2-10}\right) $, we have that%
\begin{equation}
\left\vert f\left( x\right) -\frac{1}{b-a}\int_{a}^{b}f\left( t\right)
\begin{array}{c}
_{a}d_{q}t%
\end{array}%
\right\vert  \label{2-11}
\end{equation}%
\begin{equation*}
\leq \left( b-a\right) \int_{0}^{1}\left\vert K_{q}\left( t\right)
\right\vert
\begin{array}{c}
\left\vert _{a}D_{q}f\left( tb+\left( 1-t\right) a\right) \right\vert%
\end{array}%
\begin{array}{c}
_{0}d_{q}t%
\end{array}%
\end{equation*}%
\begin{equation*}
\leq \left( b-a\right) \left[
\begin{array}{c}
\int_{0}^{\frac{x-a}{b-a}}qt%
\begin{array}{c}
\left\vert _{a}D_{q}f\left( tb+\left( 1-t\right) a\right) \right\vert%
\end{array}%
\begin{array}{c}
_{0}d_{q}t%
\end{array}
\\
+\int_{\frac{x-a}{b-a}}^{1}\left( 1-qt\right)
\begin{array}{c}
\left\vert _{a}D_{q}f\left( tb+\left( 1-t\right) a\right) \right\vert%
\end{array}%
\begin{array}{c}
_{0}d_{q}t%
\end{array}%
\end{array}%
\right]
\end{equation*}%
\begin{equation*}
\leq \left( b-a\right) \left[
\begin{array}{c}
\left( \int_{0}^{\frac{x-a}{b-a}}qt%
\begin{array}{c}
_{0}d_{q}t%
\end{array}%
\right) ^{1-\frac{1}{r}} \\
\times \left( \int_{0}^{\frac{x-a}{b-a}}qt%
\begin{array}{c}
\left\vert _{a}D_{q}f\left( tb+\left( 1-t\right) a\right) \right\vert ^{r}%
\end{array}%
\begin{array}{c}
_{0}d_{q}t%
\end{array}%
\right) ^{\frac{1}{r}} \\
+\left( \int_{\frac{x-a}{b-a}}^{1}\left( 1-qt\right)
\begin{array}{c}
_{0}d_{q}t%
\end{array}%
\right) ^{1-\frac{1}{r}} \\
\times \left( \int_{\frac{x-a}{b-a}}^{1}\left( 1-qt\right)
\begin{array}{c}
\left\vert _{a}D_{q}f\left( tb+\left( 1-t\right) a\right) \right\vert ^{r}%
\end{array}%
\begin{array}{c}
_{0}d_{q}t%
\end{array}%
\right) ^{\frac{1}{r}}%
\end{array}%
\right]
\end{equation*}%
\begin{equation*}
\leq \left( b-a\right) \left[
\begin{array}{c}
\left( \int_{0}^{\frac{x-a}{b-a}}qt%
\begin{array}{c}
_{0}d_{q}t%
\end{array}%
\right) ^{1-\frac{1}{r}} \\
\times \left( \int_{0}^{\frac{x-a}{b-a}}qt\left[
\begin{array}{c}
t\left\vert _{a}D_{q}f\left( a\right) \right\vert ^{r} \\
+\left( 1-t\right) \left\vert _{a}D_{q}f\left( b\right) \right\vert ^{r}%
\end{array}%
\right]
\begin{array}{c}
_{0}d_{q}t%
\end{array}%
\right) ^{\frac{1}{r}} \\
+\left( \int_{\frac{x-a}{b-a}}^{1}\left( 1-qt\right)
\begin{array}{c}
_{0}d_{q}t%
\end{array}%
\right) ^{1-\frac{1}{r}} \\
\times \left( \int_{\frac{x-a}{b-a}}^{1}\left( 1-qt\right) \left[
\begin{array}{c}
t\left\vert _{a}D_{q}f\left( a\right) \right\vert ^{r} \\
+\left( 1-t\right) \left\vert _{a}D_{q}f\left( b\right) \right\vert ^{r}%
\end{array}%
\right]
\begin{array}{c}
_{0}d_{q}t%
\end{array}%
\right) ^{\frac{1}{r}}%
\end{array}%
\right]
\end{equation*}%
\begin{equation*}
\leq \left( b-a\right) \left[
\begin{array}{c}
\left( \int_{0}^{\frac{x-a}{b-a}}qt%
\begin{array}{c}
_{0}d_{q}t%
\end{array}%
\right) ^{1-\frac{1}{r}} \\
\times \left(
\begin{array}{c}
\left\vert _{a}D_{q}f\left( a\right) \right\vert ^{r}\int_{0}^{\frac{x-a}{b-a%
}}qt^{2}%
\begin{array}{c}
_{0}d_{q}t%
\end{array}
\\
+\left\vert _{a}D_{q}f\left( b\right) \right\vert ^{r}\int_{0}^{\frac{x-a}{%
b-a}}qt-qt^{2}%
\begin{array}{c}
_{0}d_{q}t%
\end{array}%
\end{array}%
\right) ^{\frac{1}{r}} \\
+\left( \int_{\frac{x-a}{b-a}}^{1}\left( 1-qt\right)
\begin{array}{c}
_{0}d_{q}t%
\end{array}%
\right) ^{1-\frac{1}{r}} \\
\times \left(
\begin{array}{c}
\left\vert _{a}D_{q}f\left( a\right) \right\vert ^{r}\int_{\frac{x-a}{b-a}%
}^{1}\left( t-qt^{2}\right)
\begin{array}{c}
_{0}d_{q}t%
\end{array}
\\
+\left\vert _{a}D_{q}f\left( b\right) \right\vert ^{r}\int_{\frac{x-a}{b-a}%
}^{1}\left( 1-qt-t+qt^{2}\right)
\begin{array}{c}
_{0}d_{q}t%
\end{array}%
\end{array}%
\right) ^{\frac{1}{r}}%
\end{array}%
\right]
\end{equation*}
Using $\left( \ref{2-5}\right) $-$\left( \ref{2-8}\right) $ in $\left( \ref%
{2-11}\right) $, we obtain the desired result in $\left( \ref{2-9}\right) $.
This ends the proof.
\end{proof}

\begin{corollary}
\label{2.9}In Theorem \ref{2.8}, the following inequalities are held by the
following assumptions:

\begin{enumerate}
\item $r=1$;%
\begin{eqnarray*}
&&\left\vert f\left( x\right) -\frac{1}{b-a}\int_{a}^{b}f\left( t\right)
\begin{array}{c}
_{a}d_{q}t%
\end{array}%
\right\vert \\
&\leq &\left( b-a\right) \left[
\begin{array}{c}
\left[ \left\vert _{a}D_{q}f\left( a\right) \right\vert K_{2}\left(
a,b,x,q\right) +\left\vert _{a}D_{q}f\left( b\right) \right\vert K_{3}\left(
a,b,x,q\right) \right] \\
\left[ \left\vert _{a}D_{q}f\left( a\right) \right\vert K_{5}\left(
a,b,x,q\right) +\left\vert _{a}D_{q}f\left( b\right) \right\vert K_{6}\left(
a,b,x,q\right) \right]%
\end{array}%
\right] ,
\end{eqnarray*}

\item $r=1$ and $\left\vert _{a}D_{q}f\left( x\right) \right\vert <M$ \ for
all $x\in\left[ a,b\right] $ (a quantum Ostrowski type inequality, see \cite[%
Theorem 3.1]{NAN16});%
\begin{eqnarray*}
&&\left\vert f\left( x\right) -\frac{1}{b-a}\int_{a}^{b}f\left( t\right)
\begin{array}{c}
_{a}d_{q}t%
\end{array}%
\right\vert \\
&\leq &M\left( b-a\right) \left[ K_{2}\left( a,b,x,q\right) +K_{3}\left(
a,b,x,q\right) +K_{5}\left( a,b,x,q\right) +K_{6}\left( a,b,x,q\right) %
\right] \\
&\leq &M\left( b-a\right) \left[ K_{1}\left( a,b,x,q\right) +K_{4}\left(
a,b,x,q\right) \right] \\
&\leq &M\left( b-a\right) \left[ \frac{q}{1+q}\left( \frac{x-a}{b-a}\right)
^{2}+\frac{q}{1+q}\left( \frac{b-x}{b-a}\right) ^{2}\right] \\
&\leq &\frac{qM}{b-a}\left[ \frac{\left( x-a\right) ^{2}+\left( b-x\right)
^{2}}{1+q}\right] ,
\end{eqnarray*}

\item $r=1$, $\left\vert _{a}D_{q}f\left( x\right) \right\vert <M$ \ for all
$x\in \left[ a,b\right] $ and $q\rightarrow 1^{-}$ (Ostrowski inequality $%
\left( \ref{1-1}\right) $);

\item $r=1$ and $x=\frac{qa+b}{1+q}$ (a new quantum midpoint type
inequality);%
\begin{eqnarray*}
&&\left\vert f\left( \frac{qa+b}{1+q}\right) -\frac{1}{b-a}%
\int_{a}^{b}f\left( t\right)
\begin{array}{c}
_{a}d_{q}t%
\end{array}%
\right\vert \\
&\leq &\left( b-a\right) \left[
\begin{array}{c}
\left[ \left\vert _{a}D_{q}f\left( a\right) \right\vert K_{2}\left( a,b,%
\frac{qa+b}{1+q},q\right) +\left\vert _{a}D_{q}f\left( b\right) \right\vert
K_{3}\left( a,b,\frac{qa+b}{1+q},q\right) \right] \\
\left[ \left\vert _{a}D_{q}f\left( a\right) \right\vert K_{5}\left( a,b,%
\frac{qa+b}{1+q},q\right) +\left\vert _{a}D_{q}f\left( b\right) \right\vert
K_{6}\left( a,b,\frac{qa+b}{1+q},q\right) \right]%
\end{array}%
\right] \\
&\leq &\left( b-a\right) \left[
\begin{array}{c}
\left[ \left\vert _{a}D_{q}f\left( a\right) \right\vert \frac{q}{\left(
1+q\right) ^{3}\left( 1+q+q^{2}\right) }+\left\vert _{a}D_{q}f\left(
b\right) \right\vert \frac{q^{2}+q^{3}}{\left( 1+q\right) ^{3}\left(
1+q+q^{2}\right) }\right] \\
\left[ \left\vert _{a}D_{q}f\left( a\right) \right\vert \frac{2q}{\left(
1+q\right) ^{3}\left( 1+q+q^{2}\right) }+\left\vert _{a}D_{q}f\left(
b\right) \right\vert \frac{-2q+q^{3}+q^{4}+q^{5}}{\left( 1+q\right)
^{3}\left( 1+q+q^{2}\right) }\right]%
\end{array}%
\right] \\
&\leq &\left( b-a\right) \left[
\begin{array}{c}
\left\vert _{a}D_{q}f\left( a\right) \right\vert \frac{3q}{\left( 1+q\right)
^{3}\left( 1+q+q^{2}\right) } \\
+\left\vert _{a}D_{q}f\left( b\right) \right\vert \frac{%
-2q+q^{2}+2q^{3}+q^{4}+q^{5}}{\left( 1+q\right) ^{3}\left( 1+q+q^{2}\right) }%
\end{array}%
\right] ,
\end{eqnarray*}

\item $r=1$, $x=\frac{qa+b}{1+q}$ and $q\rightarrow 1^{-}$ (a midpoint type
inequality, see \cite[Theorem 2.2]{K04});%
\begin{equation*}
\left\vert f\left( \frac{a+b}{2}\right) -\frac{1}{b-a}\int_{a}^{b}f\left(
t\right) dt\right\vert \leq \frac{\left( b-a\right) \left[ \left\vert
f^{\prime }\left( a\right) \right\vert +\left\vert f^{\prime }\left(
b\right) \right\vert \right] }{8},
\end{equation*}

\item $r=1$ and $x=\frac{a+b}{2}$ (a new quantum midpoint type inequality);%
\begin{eqnarray*}
&&\left\vert f\left( \frac{a+b}{2}\right) -\frac{1}{b-a}\int_{a}^{b}f\left(
t\right)
\begin{array}{c}
_{a}d_{q}t%
\end{array}%
\right\vert \\
&\leq &\left( b-a\right) \left[
\begin{array}{c}
\left[ \left\vert _{a}D_{q}f\left( a\right) \right\vert K_{2}\left( a,b,%
\frac{a+b}{2},q\right) +\left\vert _{a}D_{q}f\left( b\right) \right\vert
K_{3}\left( a,b,\frac{a+b}{2},q\right) \right] \\
\left[ \left\vert _{a}D_{q}f\left( a\right) \right\vert K_{5}\left( a,b,%
\frac{a+b}{2},q\right) +\left\vert _{a}D_{q}f\left( b\right) \right\vert
K_{6}\left( a,b,\frac{a+b}{2},q\right) \right]%
\end{array}%
\right] \\
&\leq &\left( b-a\right) \left[
\begin{array}{c}
\left[ \left\vert _{a}D_{q}f\left( a\right) \right\vert \frac{q}{8\left(
1+q+q^{2}\right) }+\left\vert _{a}D_{q}f\left( b\right) \right\vert \frac{%
q+q^{2}+2q^{3}}{8\left( 1+q\right) \left( 1+q+q^{2}\right) }\right] \\
\left[ \left\vert _{a}D_{q}f\left( a\right) \right\vert \frac{6-q-q^{2}}{%
8\left( 1+q\right) \left( 1+q+q^{2}\right) }+\left\vert _{a}D_{q}f\left(
b\right) \right\vert \frac{3q+3q^{2}+2q^{3}-6}{8\left( 1+q\right) \left(
1+q+q^{2}\right) }\right]%
\end{array}%
\right] \\
&\leq &\left( b-a\right) \left[
\begin{array}{c}
\left\vert _{a}D_{q}f\left( a\right) \right\vert \frac{6}{8\left( 1+q\right)
\left( 1+q+q^{2}\right) } \\
+\left\vert _{a}D_{q}f\left( b\right) \right\vert \frac{4q+4q^{2}+4q^{3}-6}{%
8\left( 1+q\right) \left( 1+q+q^{2}\right) }%
\end{array}%
\right] ,
\end{eqnarray*}

\item $\left\vert _{a}D_{q}f\left( x\right) \right\vert <M$ \ for all $x\in %
\left[ a,b\right] $ (a quantum Ostrowski type inequality, see \cite[Theorem
3.1]{NAN16});%
\begin{eqnarray*}
&&\left\vert f\left( x\right) -\frac{1}{b-a}\int_{a}^{b}f\left( t\right)
\begin{array}{c}
_{a}d_{q}t%
\end{array}%
\right\vert  \\
&\leq &\left( b-a\right) M\left[
\begin{array}{c}
K_{1}^{1-\frac{1}{r}}\left( a,b,x,q\right) \left[ K_{2}\left( a,b,x,q\right)
+K_{3}\left( a,b,x,q\right) \right] ^{\frac{1}{r}} \\
+K_{4}^{1-\frac{1}{r}}\left( a,b,x,q\right) \left[ K_{5}\left(
a,b,x,q\right) +K_{6}\left( a,b,x,q\right) \right] ^{\frac{1}{r}}%
\end{array}%
\right]  \\
&\leq &\left( b-a\right) M\left[
\begin{array}{c}
K_{1}^{1-\frac{1}{r}}\left( a,b,x,q\right) K_{1}^{\frac{1}{r}}\left(
a,b,x,q\right)  \\
+K_{4}^{1-\frac{1}{r}}\left( a,b,x,q\right) K_{4}^{\frac{1}{r}}\left(
a,b,x,q\right)
\end{array}%
\right]  \\
&\leq &\left( b-a\right) M\left[ K_{1}\left( a,b,x,q\right) +K_{4}\left(
a,b,x,q\right) \right]  \\
&\leq &M\left( b-a\right) \left[ \frac{q}{1+q}\left( \frac{x-a}{b-a}\right)
^{2}+\frac{q}{1+q}\left( \frac{b-x}{b-a}\right) ^{2}\right]  \\
&\leq &\frac{qM}{b-a}\left[ \frac{\left( x-a\right) ^{2}+\left( b-x\right)
^{2}}{1+q}\right] ,
\end{eqnarray*}

\item $x=\frac{qa+b}{1+q}$ (a new quantum midpoint type inequality);%
\begin{eqnarray*}
&&\left\vert f\left( \frac{qa+b}{1+q}\right) -\frac{1}{b-a}%
\int_{a}^{b}f\left( t\right)
\begin{array}{c}
_{a}d_{q}t%
\end{array}%
\right\vert  \\
&\leq &\left( b-a\right) \left[
\begin{array}{c}
K_{1}^{1-\frac{1}{r}}\left( a,b,\frac{qa+b}{1+q},q\right) \left[
\begin{array}{c}
\left\vert _{a}D_{q}f\left( a\right) \right\vert ^{r}K_{2}\left( a,b,\frac{%
qa+b}{1+q},q\right)  \\
+\left\vert _{a}D_{q}f\left( b\right) \right\vert ^{r}K_{3}\left( a,b,\frac{%
qa+b}{1+q},q\right)
\end{array}%
\right] ^{\frac{1}{r}} \\
+K_{4}^{1-\frac{1}{r}}\left( a,b,\frac{qa+b}{1+q},q\right) \left[
\begin{array}{c}
\left\vert _{a}D_{q}f\left( a\right) \right\vert ^{r}K_{5}\left( a,b,\frac{%
qa+b}{1+q},q\right)  \\
+\left\vert _{a}D_{q}f\left( b\right) \right\vert ^{r}K_{6}\left( a,b,\frac{%
qa+b}{1+q},q\right)
\end{array}%
\right] ^{\frac{1}{r}}%
\end{array}%
\right]  \\
&\leq &\left( b-a\right) \left[
\begin{array}{c}
\left[ \frac{q}{\left( 1+q\right) ^{3}}\right] ^{1-\frac{1}{r}}\left[
\begin{array}{c}
\left\vert _{a}D_{q}f\left( a\right) \right\vert ^{r}\frac{q}{\left(
1+q\right) ^{3}\left( 1+q+q^{2}\right) } \\
+\left\vert _{a}D_{q}f\left( b\right) \right\vert ^{r}\frac{q^{2}+q^{3}}{%
\left( 1+q\right) ^{3}\left( 1+q+q^{2}\right) }%
\end{array}%
\right] ^{\frac{1}{r}} \\
+\left[ \frac{q^{3}}{\left( 1+q\right) ^{3}}\right] ^{1-\frac{1}{r}}\left[
\begin{array}{c}
\left\vert _{a}D_{q}f\left( a\right) \right\vert ^{r}\frac{2q}{\left(
1+q\right) ^{3}\left( 1+q+q^{2}\right) } \\
+\left\vert _{a}D_{q}f\left( b\right) \right\vert ^{r}\frac{%
-2q+q^{3}+q^{4}+q^{5}}{\left( 1+q\right) ^{3}\left( 1+q+q^{2}\right) }%
\end{array}%
\right] ^{\frac{1}{r}}%
\end{array}%
\right] ,
\end{eqnarray*}

\item $x=\frac{qa+b}{1+q}$ and $q\rightarrow 1^{-}$ (a midpoint type
inequality, see \cite[Corollary 17]{ASKI16});%
\begin{eqnarray*}
&&\left\vert f\left( \frac{a+b}{2}\right) -\frac{1}{b-a}\int_{a}^{b}f\left(
t\right) dt\right\vert  \\
&\leq&\left( b-a\right) \frac{1}{2^{3-\frac{3}{r}}}\left[
\begin{array}{c}
\Big( \left\vert f^{\prime
}\left( a\right) \right\vert ^{r}\frac{1}{24}+\left\vert f^{\prime }\left(
b\right) \right\vert ^{r}\frac{1}{12}\Big) ^{\frac{1}{r}} \\
+ \Big( \left\vert f^{\prime
}\left( a\right) \right\vert ^{r}\frac{1}{12}+\left\vert f^{\prime }\left(
b\right) \right\vert ^{r}\frac{1}{24}\Big) ^{\frac{1}{r}}%
\end{array}%
\right],
\end{eqnarray*}

\item $x=\frac{a+b}{2}$ (a new quantum midpoint type inequality);%
\begin{eqnarray*}
&&\left\vert f\left( \frac{a+b}{2}\right) -\frac{1}{b-a}\int_{a}^{b}f\left(
t\right)
\begin{array}{c}
_{a}d_{q}t%
\end{array}%
\right\vert  \\
&\leq &\left( b-a\right) \left[
\begin{array}{c}
K_{1}^{1-\frac{1}{r}}\left( a,b,\frac{a+b}{2},q\right) \left[
\begin{array}{c}
\left\vert _{a}D_{q}f\left( a\right) \right\vert ^{r}K_{2}\left( a,b,\frac{%
a+b}{2},q\right)  \\
+\left\vert _{a}D_{q}f\left( b\right) \right\vert ^{r}K_{3}\left( a,b,\frac{%
a+b}{2},q\right)
\end{array}%
\right] ^{\frac{1}{r}} \\
+K_{4}^{1-\frac{1}{r}}\left( a,b,\frac{a+b}{2},q\right) \left[
\begin{array}{c}
\left\vert _{a}D_{q}f\left( a\right) \right\vert ^{r}K_{5}\left( a,b,\frac{%
a+b}{2},q\right)  \\
+\left\vert _{a}D_{q}f\left( b\right) \right\vert ^{r}K_{6}\left( a,b,\frac{%
a+b}{2},q\right)
\end{array}%
\right] ^{\frac{1}{r}}%
\end{array}%
\right]  \\
&\leq &\left( b-a\right) \left( \frac{q}{4\left( 1+q\right) }\right) ^{1-%
\frac{1}{r}}\left[
\begin{array}{c}
\left[
\begin{array}{c}
\left\vert _{a}D_{q}f\left( a\right) \right\vert ^{r}\frac{q}{8\left(
1+q+q^{2}\right) } \\
+\left\vert _{a}D_{q}f\left( b\right) \right\vert ^{r}\frac{q+q^{2}+2q^{3}}{%
8\left( 1+q\right) \left( 1+q+q^{2}\right) }%
\end{array}%
\right] ^{\frac{1}{r}} \\
+\left[
\begin{array}{c}
\left\vert _{a}D_{q}f\left( a\right) \right\vert ^{r}\frac{6-q-q^{2}}{%
8\left( 1+q\right) \left( 1+q+q^{2}\right) } \\
+\left\vert _{a}D_{q}f\left( b\right) \right\vert ^{r}\frac{%
3q+3q^{2}+2q^{3}-6}{8\left( 1+q\right) \left( 1+q+q^{2}\right) }%
\end{array}%
\right] ^{\frac{1}{r}}%
\end{array}%
\right] ,
\end{eqnarray*}

\item $x=\frac{a+qb}{1+q}$ (a new quantum midpoint type inequality);%
\begin{eqnarray*}
&&\left\vert f\left( \frac{a+qb}{1+q}\right) -\frac{1}{b-a}%
\int_{a}^{b}f\left( t\right)
\begin{array}{c}
_{a}d_{q}t%
\end{array}%
\right\vert  \\
&\leq &\left( b-a\right) \left[
\begin{array}{c}
K_{1}^{1-\frac{1}{r}}\left( a,b,\frac{a+qb}{1+q},q\right) \left[
\begin{array}{c}
\left\vert _{a}D_{q}f\left( a\right) \right\vert ^{r}K_{2}\left( a,b,\frac{%
a+qb}{1+q},q\right)  \\
+\left\vert _{a}D_{q}f\left( b\right) \right\vert ^{r}K_{3}\left( a,b,\frac{%
a+qb}{1+q},q\right)
\end{array}%
\right] ^{\frac{1}{r}} \\
+K_{4}^{1-\frac{1}{r}}\left( a,b,\frac{a+qb}{1+q},q\right) \left[
\begin{array}{c}
\left\vert _{a}D_{q}f\left( a\right) \right\vert ^{r}K_{5}\left( a,b,\frac{%
a+qb}{1+q},q\right)  \\
+\left\vert _{a}D_{q}f\left( b\right) \right\vert ^{r}K_{6}\left( a,b,\frac{%
a+qb}{1+q},q\right)
\end{array}%
\right] ^{\frac{1}{r}}%
\end{array}%
\right]  \\
&\leq &\left( b-a\right) \left[
\begin{array}{c}
\left[ \frac{q^{3}}{\left( 1+q\right) ^{3}}\right] ^{1-\frac{1}{r}}\left[
\begin{array}{c}
\left\vert _{a}D_{q}f\left( a\right) \right\vert ^{r}\frac{q^{4}}{\left(
1+q\right) ^{3}\left( 1+q+q^{2}\right) } \\
+\left\vert _{a}D_{q}f\left( b\right) \right\vert ^{r}\frac{q^{3}+q^{5}}{%
\left( 1+q\right) ^{3}\left( 1+q+q^{2}\right) }%
\end{array}%
\right] ^{\frac{1}{r}} \\
+\left[ \frac{q}{\left( 1+q\right) ^{3}}\right] ^{1-\frac{1}{r}}\left[
\begin{array}{c}
\left\vert _{a}D_{q}f\left( a\right) \right\vert ^{r}\frac{1+2q-q^{3}}{%
\left( 1+q\right) ^{3}\left( 1+q+q^{2}\right) } \\
+\left\vert _{a}D_{q}f\left( b\right) \right\vert ^{r}\frac{-1-q+q^{2}+2q^{3}%
}{\left( 1+q\right) ^{3}\left( 1+q+q^{2}\right) }%
\end{array}%
\right] ^{\frac{1}{r}}%
\end{array}%
\right] .
\end{eqnarray*}
\end{enumerate}
\end{corollary}

Finally, we give the following calculated quantum definite integrals used as
the next Theorem 4.
%The following calculations of quantum definite integrals are used in next result:%
\begin{eqnarray}
\int_{0}^{\frac{x-a}{b-a}}t%
\begin{array}{c}
_{0}d_{q}t%
\end{array}
&=&\left( 1-q\right) \frac{x-a}{b-a}\sum_{n=0}^{\infty }q^{n}\left( q^{n}%
\frac{x-a}{b-a}\right)  \label{2-12} \\
&=&\left( 1-q\right) \left( \frac{x-a}{b-a}\right) ^{2}\frac{1}{1-q^{2}}
\notag \\
&=&\frac{1}{1+q}\left( \frac{x-a}{b-a}\right) ^{2},  \notag
\end{eqnarray}%
\begin{eqnarray}
\int_{0}^{\frac{x-a}{b-a}}\left( 1-t\right)
\begin{array}{c}
_{0}d_{q}t%
\end{array}
&=&\left( 1-q\right) \frac{x-a}{b-a}\sum_{n=0}^{\infty }q^{n}\left( 1-q^{n}%
\frac{x-a}{b-a}\right)  \label{2-13} \\
&=&\left( 1-q\right) \frac{x-a}{b-a}\left( \frac{1}{1-q}-\left( \frac{x-a}{%
b-a}\right) \frac{1}{1-q^{2}}\right)  \notag \\
&=&\frac{x-a}{b-a}\left( 1-\frac{1}{1+q}\left( \frac{x-a}{b-a}\right) \right)
\notag \\
&=&\frac{x-a}{b-a}-\frac{1}{1+q}\left( \frac{x-a}{b-a}\right) ^{2},  \notag
\end{eqnarray}%
\begin{eqnarray}
\int_{\frac{x-a}{b-a}}^{1}t%
\begin{array}{c}
_{0}d_{q}t%
\end{array}
&=&\int_{0}^{1}t%
\begin{array}{c}
_{0}d_{q}t%
\end{array}%
-\int_{0}^{\frac{x-a}{b-a}}t%
\begin{array}{c}
_{0}d_{q}t%
\end{array}
\label{2-14} \\
&=&\frac{1}{1+q}-\frac{1}{1+q}\left( \frac{x-a}{b-a}\right) ^{2}  \notag \\
&=&\frac{1}{1+q}\left( 1-\left( \frac{x-a}{b-a}\right) ^{2}\right) ,  \notag
\end{eqnarray}%
and
\begin{eqnarray}
\int_{\frac{x-a}{b-a}}^{1}\left( 1-t\right)
\begin{array}{c}
_{0}d_{q}t%
\end{array}
&=&\int_{0}^{1}\left( 1-t\right)
\begin{array}{c}
_{0}d_{q}t%
\end{array}%
-\int_{0}^{\frac{x-a}{b-a}}\left( 1-t\right)
\begin{array}{c}
_{0}d_{q}t%
\end{array}
\label{2-15} \\
&=&\frac{q}{1+q}-\frac{x-a}{b-a}+\frac{1}{1+q}\left( \frac{x-a}{b-a}\right)
^{2}.  \notag
\end{eqnarray}

\begin{theorem}
\label{2.10} Let $f:\left[ a,b\right] \rightarrow
%TCIMACRO{\U{211d} }%
%BeginExpansion
\mathbb{R}
%EndExpansion
$ be an arbitrary function with $%
\begin{array}{c}
_{a}D_{q}f%
\end{array}%
$ is quantum integrable on $\left[ a,b\right] $. If $%
\begin{array}{c}
\left\vert _{a}D_{q}f\right\vert ^{r},%
\end{array}%
$ $r>1$ and $\frac{1}{r}+\frac{1}{p}=1$ is a convex function, then the
following quantum integral inequality holds:%
\begin{equation}
\left\vert f\left( x\right) -\frac{1}{b-a}\int_{a}^{b}f\left( t\right)
\begin{array}{c}
_{a}d_{q}t%
\end{array}%
\right\vert  \label{2-16}
\end{equation}%
\begin{equation*}
\leq \left( b-a\right) \left[
\begin{array}{c}
\left( \int_{0}^{\frac{x-a}{b-a}}qt%
\begin{array}{c}
_{0}d_{q}t%
\end{array}%
\right) ^{\frac{1}{p}} \\
\times \left(
\begin{array}{c}
\left\vert _{a}D_{q}f\left( a\right) \right\vert ^{r}\left[ \frac{1}{1+q}%
\left( \frac{x-a}{b-a}\right) ^{2}\right] \\
+\left\vert _{a}D_{q}f\left( b\right) \right\vert ^{r}\left[ \frac{x-a}{b-a}-%
\frac{1}{1+q}\left( \frac{x-a}{b-a}\right) ^{2}\right]%
\end{array}%
\right) ^{\frac{1}{r}} \\
+\left( \int_{\frac{x-a}{b-a}}^{1}\left( 1-qt\right) ^{p}%
\begin{array}{c}
_{0}d_{q}t%
\end{array}%
\right) ^{\frac{1}{p}} \\
\times \left(
\begin{array}{c}
\left\vert _{a}D_{q}f\left( a\right) \right\vert ^{r}\left[ \frac{1}{1+q}%
\left( 1-\left( \frac{x-a}{b-a}\right) ^{2}\right) \right] \\
+\left\vert _{a}D_{q}f\left( b\right) \right\vert ^{r}\left[ \frac{q}{1+q}-%
\frac{x-a}{b-a}+\frac{1}{1+q}\left( \frac{x-a}{b-a}\right) ^{2}\right]%
\end{array}%
\right) ^{\frac{1}{r}}%
\end{array}%
\right]
\end{equation*}%
for all $x\in \left[ a,b\right] $.
\end{theorem}

\begin{proof}
By using Lemma \ref{2.6}, quantum H\"{o}lder inequality and $\left( \ref{2-6}%
\right) $, we have that%
\begin{equation}
\left\vert f\left( x\right) -\frac{1}{b-a}\int_{a}^{b}f\left( t\right)
\begin{array}{c}
_{a}d_{q}t%
\end{array}%
\right\vert  \label{2-17}
\end{equation}%
\begin{equation*}
\leq \left( b-a\right) \int_{0}^{1}\left\vert K_{q}\left( t\right)
\right\vert
\begin{array}{c}
\left\vert _{a}D_{q}f\left( tb+\left( 1-t\right) a\right) \right\vert%
\end{array}%
\begin{array}{c}
_{0}d_{q}t%
\end{array}%
\end{equation*}%
\begin{equation*}
\leq \left( b-a\right) \left[
\begin{array}{c}
\int_{0}^{\frac{x-a}{b-a}}qt%
\begin{array}{c}
\left\vert _{a}D_{q}f\left( tb+\left( 1-t\right) a\right) \right\vert%
\end{array}%
\begin{array}{c}
_{0}d_{q}t%
\end{array}
\\
+\int_{\frac{x-a}{b-a}}^{1}\left( 1-qt\right)
\begin{array}{c}
\left\vert _{a}D_{q}f\left( tb+\left( 1-t\right) a\right) \right\vert%
\end{array}%
\begin{array}{c}
_{0}d_{q}t%
\end{array}%
\end{array}%
\right]
\end{equation*}%
\begin{equation*}
\leq \left( b-a\right) \left[
\begin{array}{c}
\left( \int_{0}^{\frac{x-a}{b-a}}\left( qt\right) ^{p}%
\begin{array}{c}
_{0}d_{q}t%
\end{array}%
\right) ^{\frac{1}{p}} \\
\times \left( \int_{0}^{\frac{x-a}{b-a}}%
\begin{array}{c}
\left\vert _{a}D_{q}f\left( tb+\left( 1-t\right) a\right) \right\vert ^{r}%
\end{array}%
\begin{array}{c}
_{0}d_{q}t%
\end{array}%
\right) ^{\frac{1}{r}} \\
+\left( \int_{\frac{x-a}{b-a}}^{1}\left( 1-qt\right) ^{p}%
\begin{array}{c}
_{0}d_{q}t%
\end{array}%
\right) ^{\frac{1}{p}} \\
\times \left( \int_{\frac{x-a}{b-a}}^{1}%
\begin{array}{c}
\left\vert _{a}D_{q}f\left( tb+\left( 1-t\right) a\right) \right\vert ^{r}%
\end{array}%
\begin{array}{c}
_{0}d_{q}t%
\end{array}%
\right) ^{\frac{1}{r}}%
\end{array}%
\right]
\end{equation*}%
\begin{equation*}
\leq \left( b-a\right) \left[
\begin{array}{c}
\left( \int_{0}^{\frac{x-a}{b-a}}\left( qt\right) ^{p}%
\begin{array}{c}
_{0}d_{q}t%
\end{array}%
\right) ^{\frac{1}{p}} \\
\times \left( \int_{0}^{\frac{x-a}{b-a}}\left[ t\left\vert _{a}D_{q}f\left(
a\right) \right\vert ^{r}+\left( 1-t\right) \left\vert _{a}D_{q}f\left(
b\right) \right\vert ^{r}\right]
\begin{array}{c}
_{0}d_{q}t%
\end{array}%
\right) ^{\frac{1}{r}} \\
+\left( \int_{\frac{x-a}{b-a}}^{1}\left( 1-qt\right) ^{p}%
\begin{array}{c}
_{0}d_{q}t%
\end{array}%
\right) ^{\frac{1}{p}} \\
\times \left( \int_{\frac{x-a}{b-a}}^{1}\left[ t\left\vert _{a}D_{q}f\left(
a\right) \right\vert ^{r}+\left( 1-t\right) \left\vert _{a}D_{q}f\left(
b\right) \right\vert ^{r}\right]
\begin{array}{c}
_{0}d_{q}t%
\end{array}%
\right) ^{\frac{1}{r}}%
\end{array}%
\right]
\end{equation*}%
\begin{equation*}
\leq \left( b-a\right) \left[
\begin{array}{c}
\left( \int_{0}^{\frac{x-a}{b-a}}qt%
\begin{array}{c}
_{0}d_{q}t%
\end{array}%
\right) ^{\frac{1}{p}} \\
\times \left( \left\vert _{a}D_{q}f\left( a\right) \right\vert ^{r}\int_{0}^{%
\frac{x-a}{b-a}}t%
\begin{array}{c}
_{0}d_{q}t%
\end{array}%
+\left\vert _{a}D_{q}f\left( b\right) \right\vert ^{r}\int_{0}^{\frac{x-a}{%
b-a}}\left( 1-t\right)
\begin{array}{c}
_{0}d_{q}t%
\end{array}%
\right) ^{\frac{1}{r}} \\
+\left( \int_{\frac{x-a}{b-a}}^{1}\left( 1-qt\right) ^{p}%
\begin{array}{c}
_{0}d_{q}t%
\end{array}%
\right) ^{\frac{1}{p}} \\
\times \left( \left\vert _{a}D_{q}f\left( a\right) \right\vert ^{r}\int_{%
\frac{x-a}{b-a}}^{1}t%
\begin{array}{c}
_{0}d_{q}t%
\end{array}%
+\left\vert _{a}D_{q}f\left( b\right) \right\vert ^{r}\int_{\frac{x-a}{b-a}%
}^{1}\left( 1-t\right)
\begin{array}{c}
_{0}d_{q}t%
\end{array}%
\right) ^{\frac{1}{r}}%
\end{array}%
\right] .
\end{equation*}
Using $\left( \ref{2-12}\right) $-$\left( \ref{2-15}\right) $ in $\left( \ref%
{2-17}\right) $, we obtain the desired result in $\left( \ref{2-16}\right) $%
. This ends the proof.
\end{proof}

\begin{remark}
In Theorem \ref{2.10}, many different inequalities could be derived
similarly to Corollary \ref{2.9}.
\end{remark}

\section{Conclusion}
Utilizing mappings whose first derivatives absolute values are quantum differentiable convex, we establish some  quantum integral inequalities of Ostrowski type in terms of the discovered  quantum Montgomery identity.  Furthermore, we investigate the important relevant connections between the results obtained in this work with those introduced in earlier published papers. Many  sub-results can be derived from our main results by considering the special variable value for $x\in[a,b]$,  some fixed  value for $r$, as well as $q\rightarrow 1^{-}$. It is worthwhile to mention that certain quantum inequalities  presented in this work generalize parts of the very recent results given by Alp et al. (2018) and Noor et al. (2016).  With these contributions, we hope to motivate the interested researchers to explore this fascinating field of the quantum integral inequality based on the techniques and ideas developed in this article.\\

\noindent\textbf{Acknowledgements}
The first author would like to thank Ondokuz May\i s University for being a
visiting professor and providing excellent research facilities.\\

\noindent\textbf{Competing interests }\\
The authors declare that they have no competing interests.

\noindent\textbf{Authors' contributions}\\
All authors contributed equally to the writing of this paper. All authors
read and approved the final manuscript.\\


\begin{thebibliography}{99}
\bibitem{A14} A. A. Aljinovi\'{c}, Montgomery identity and Ostrowski type
inequalities for Riemann-Liouville fractional integral, J. Math., Article ID
503195 (2014) 1-6.

\bibitem{A95} G. A. Anastassiou, Ostrowski type inequalities, Proc. Amer.
Math. Soc., 123 (12) (1995) 3775-3781.

\bibitem{AM12} M. H. Annaby, Z. S. Mansour, $q$- Fractional Calculus and
Equations, Springer, Heidelberg, (2012).

\bibitem{ADDC10} M. Alomari, M. Darus, S. S. Dragomir, P. Cerone, Ostrowski
type inequalities for functions whose derivatives are $s$-convex in the
second sense, Appl. Math. Lett., 23 (2010) 1071-1076.

\bibitem{AS17} N. Alp, M. Z. Sar\i kaya, A new definition and properties of
quantum integral which calls $\overline{q}$-integral, Konuralp J. Math., 5
(2) (2017) 146-159.

\bibitem{ASKI16} N. Alp, M. Z. Sar\i kaya , M. Kunt, \.{I}. \.{I}\c{s}can, ~$%
q$-Hermite--Hadamard inequalities and quantum estimates for midpoint type
inequalities via convex and quasi-convex functions, J. King Saud Univ. Sci.,
30 (2) (2018) 193-203.

\bibitem{BS16} H. Budak, M. Z. Sar\i kaya, On generalized Ostrowski-type
inequalities for functions whose firrst derivatives absolute values are
convex, Turkish J. Math., 40 (2016) 1193-1210.

\bibitem{CD03} P. Cerone, S. S. Dragomir, On some inequalities arising from
montgomery's identity (montgomery's identity), J. Comput. Anal. Appl., 5 (4)
(2003)\ 341-367.

\bibitem{DR02} S. S. Dragomir, T. M. Rassias, Ostrowski type inequalities
and applications in numerical integration, Kluwer Academic Publishers, 2002.

\bibitem{F17} G. Farid, Some new Ostrowski type inequalities via fractional
integrals, Int. J. Anal. Appl., 14 (2017) 64-68.

\bibitem{I16} \.{I}. \.{I}\c{s}can, Ostrowski type inequalities for $p$%
-convex functions, New trends Math. Sci., 4(3) (2016) 140-150.

\bibitem{KC01} V. Kac, P. Cheung: Quantum calculus, Springer (2001).

\bibitem{KOA11} H. Kavurmac\i , M. E. \"{O}zdemir, M. Avc\i , New Ostrowski
type inequalities for $m$-convex functions and applications, Hacet. J. Math.
Stat., 40 (2) (2011) 135-145.

\bibitem{KS15} M. E. Kiri\c{s}, M. Z. Sar\i kaya, On Ostrowski type
inequalities and \v{C}eby\v{s}ev type inequalities with applications,
Filomat, 29 (8) (2015) 1695-1713.

\bibitem{K04} U. S. Kirmaci, Inequalities for differentiable mappings and
applications to special means of real numbers and to midpoint formula, Appl.
Math. Comput., 147 (2004) 137-146.

\bibitem{KIAS18} M. Kunt , \.{I}. \.{I}\c{s}can, N. Alp, M. Z. Sar\i kaya, $%
\left( p,q\right) $-Hermite--Hadamard inequalities and $\left( p,q\right) $%
-estimates for midpoint type inequalities via convex and quasi-convex
functions, Rev. R. Acad. Cienc. Exactas F\'{\i}s. Nat. Ser. A Math., 112
(2018) 969-992.

\bibitem{KLID19} M. Kunt, M. A. Latif, \.{I}. \.{I}\c{s}can, S. S. Dragomir,
Quantum Hermite-Hadamard type inequality and some estimates of quantum
midpoint type inequalities for double integrals, Sigma J. Eng. Nat. Sci., 37
(1) (2019) 207-223.

\bibitem{Liu2017} W. J. Liu, H. F. Zhuang, Some quantum estimates of
Hermite--Hadamard inequalities for convex functions, J. Appl. Anal. Comput.,
7 (2) (2017) 501--522.

\bibitem{L08} Z. Liu, Some Ostrowski type inequalities, Math. Comput.
Model., 48 (2008) 949-960.

\bibitem{M14} M. Mat{\l }oka, Ostrowski type inequalities for functions
whose derivatives are $h$-convex via fractional integrals, J. Sci. Res. \&
Rep., 3(12) (2014) 1633-1641.

\bibitem{MPF91} D. S. Mitrinovi\'{c}, J. E. Pe\v{c}ari\'{c}, and A. M. Fink,
Inequalities for functions and their integrals and derivatives, Kluwer
Academic, Dordrecht, 1991.

\bibitem{NAN16} M. A. Noor, M. U. Awan, K. I. Noor, Quantum Ostrowski
inequalities for $q$-differentiable convex functions, J. Math. Inequal., 10
(4) (2016) 1013-1018.

\bibitem{NNA15} M. A. Noor, K. I. Noor, M. U. Awan, Some quantum estimates
for Hermite-Hadamard inequalities, Appl. Math. Comput. 251 (2015) 675--679.

\bibitem{NNA15a} M. A. Noor, K. I. Noor, M. U. Awan, Some quantum integral
inequalities via preinvex functions, Appl. Math. Comput. 269 (2015) 242--251.

\bibitem{O38} A. Ostrowski, \"{U}ber die Absolutabweichung einer
differentienbaren Funktionen von ihren Integralmittelwert. Comment. Math.
Hel, 10 (1938), 226--227.

\bibitem{OKA14} M. E. \"{O}zdemir, H. Kavurmac\i , M. Avc\i , Ostrowski type
inequalities for convex functions, Tamkang J. Math., 45 (4) (2014) 335-340.

\bibitem{SB17} M. Z. Sar\i kaya, H. Budak, Generalized Ostrowski type
inequalities for local fractional integrals, Proc. Amer. Math. Soc., 145(4)
(2017) 1527-1538.

\bibitem{SOS12} E. Set, M. E. \"{O}zdemir, M. Z. Sar\i kaya, New
inequalities of Ostrowski's type for $s$-convex functions in the second
sense with applications, Facta Univ. Ser. Math. Inform., 27 (1) (2012) 67-82.

\bibitem{SNT15} W. Sudsutad, S. K. Ntouyas, J. Tariboon, Quantum integral
inequalities for convex functions, J. Math. Inequal., 9 (3) (2015) 781--793.

\bibitem{TN13} J. Tariboon, S. K. Ntouyas, Quantum calculus on finite
intervals and applications to impulsive difference equations, Adv.
Difference Equ. 282 (2013) 1-19.

\bibitem{TN14} J. Tariboon, S. K. Ntouyas, Quantum integral inequalities on
finite intervals, J. Inequal. Appl. Article ID 121 (2014) 1-13.

\bibitem{TGB18} M. Tun\c{c}, E. G\"{o}v, S. Balge\c{c}ti, Simpson type
quantum integral inequalities for convex functions, Miskolc Math. Notes, 19
(1) (2018) 649-664.

\bibitem{ZDWS18} Y. Zhang , T.-S. Du, H. Wang , Y.-J. Shen, Different types
of quantum integral inequalities via $\left( \alpha ,m\right) $-convexity,
J. Inequal. Appl., Article ID 264 (2018) 1-24.

\bibitem{ZLP19} H. F. Zhuang, W. J. Liu, J. Park, Some quantum estimates of
Hermite-Hadamard inequalities for quasi-convex functions, Mathematics,
Article ID 152, 7 (2) (2019) 1-18.
\end{thebibliography}
\end{document}